\documentclass[12pt]{amsart}
\usepackage{amssymb,amsmath}

\newif\iflong
\longtrue

\iflong
\long\def\bpf#1\epf{\begin{proof}#1\end{proof}}
\else
\long\def\bpf#1\epf{} 
\fi

\newcommand{\lcm}{\operatorname{lcm}}
\newcommand{\mytab}[3]{\begin{table}[!htp]\caption{#1}\label{#2}\begin{tabular}{|c||r|r|r|r|r|r|}
\hline $n$ & 8 & 16 & 32 & 64 & 128 & 256\\ \hline\hline #3 \hline \end{tabular} \end{table}}

\newcommand{\vphi}{\varphi}

\newcommand{\len}{\operatorname{len}}
\newcommand{\op}[2]{\bigl(#1,#2\bigr)}
\newcommand{\act}[2]{{{}^{#1}\!{#2}}}

\newcommand{\x}{\times}

\newcommand{\F}{\mathbb{F}}
\newcommand{\sbst}{\subseteq}

      \newenvironment{changemargin}[2]{\begin{list}{}{
         \setlength{\topsep}{0pt}\setlength{\leftmargin}{0pt}
         \setlength{\rightmargin}{0pt}
         \setlength{\listparindent}{\parindent}
         \setlength{\itemindent}{\parindent}
         \setlength{\parsep}{0pt plus 1pt}
         \addtolength{\leftmargin}{#1}\addtolength{\rightmargin}{#2}
         }\item }{\end{list}}

\newtheorem{thm}{Theorem}
\newtheorem{prop}[thm]{Proposition}
\newtheorem{cor}[thm]{Corollary}
\newcommand{\bcor}{\begin{cor}}
\newcommand{\ecor}{\end{cor}}
\newtheorem{lem}[thm]{Lemma}
\newtheorem{conj}[thm]{Conjecture}

\newtheorem{notat}[thm]{Notation}
\newtheorem{prob}[thm]{Problem}
\theoremstyle{definition}
\newtheorem{defn}[thm]{Definition}
\newtheorem{alg}[thm]{Algorithm}
\theoremstyle{remark}
\newtheorem{rem}[thm]{Remark}
\newtheorem{as}[thm]{Assumption}
\newtheorem{exa}[thm]{Example}

\newcommand{\inv}{^{-1}}
\newcommand{\mx}[1]{\left (\begin{matrix}#1\end{matrix}\right )}
\newcommand{\be}{\begin{enumerate}}
\newcommand{\ee}{\end{enumerate}}
\newcommand{\bi}{\begin{itemize}}
\newcommand{\ei}{\end{itemize}}
\newcommand{\itm}{\item}
\newcommand{\ble}{\begin{lem}}
\newcommand{\ele}{\end{lem}}
\newcommand{\bth}{\begin{thm}}
\newcommand{\bpr}{\begin{prop}}
\newcommand{\epr}{\end{prop}}
\newcommand{\bco}{\begin{cor}}
\newcommand{\eco}{\end{cor}}
\newcommand{\bcon}{\begin{conj}}
\newcommand{\econ}{\end{conj}}
\newcommand{\bde}{\begin{defn}}
\newcommand{\ede}{\end{defn}}
\newcommand{\bex}{\begin{exa}}
\newcommand{\eex}{\end{exa}}
\newcommand{\brem}{\begin{rem}}
\newcommand{\erem}{\end{rem}}
\newcommand{\bnot}{\begin{notat}}
\newcommand{\enot}{\end{notat}}
\newcommand{\balg}{\begin{alg}}
\newcommand{\ealg}{\end{alg}}

\newcommand{\cI}{\mathcal{I}}

\long\def\forget#1\forgotten{} %

\begin{document}
\title[Short expressions of permutations and cryptanalysis]{Short expressions of permutations as products and
cryptanalysis of the Algebraic Eraser}

\author{Arkadius Kalka, Mina Teicher, and Boaz Tsaban}
\address{Department of Mathematics,
Bar Ilan University,
Ramat Gan 52900,
Israel}
\email{Arkadius.Kalka@rub.de, teicher@math.biu.ac.il, tsaban@math.biu.ac.il}
\urladdr{http://www.cs.biu.ac.il/\~{}tsaban}

\begin{abstract}
On March 2004, Anshel, Anshel, Goldfeld, and Lem\-ieux introduced
the \emph{Algebraic Eraser} scheme for key agreement over an
insecure channel, using a novel hybrid of infinite and
finite noncommutative groups. They also introduced the
\emph{Colored Burau Key Agreement Protocol (CBKAP)}, a concrete
realization of this scheme.

We present general, efficient heuristic algorithms, which extract the shared key
out of the public information provided by CBKAP.
These algorithms are, according to heuristic reasoning and according to massive experiments,
successful for all sizes of the security parameters,
assuming that the keys are chosen with standard distributions.

Our methods come from probabilistic group theory (permutation group actions and expander graphs).
In particular, we provide a simple
algorithm for finding short expressions of permutations in $S_n$,
as products of given random permutations.
Heuristically, our algorithm gives expressions of length $O(n^2\log n)$,
in time and space $O(n^3)$. Moreover, this is provable from \emph{the Minimal Cycle Conjecture},
a simply stated hypothesis concerning the uniform distribution on $S_n$.
Experiments show that the constants in these estimations
are small. This is the first practical algorithm for this problem for $n\ge 256$.

\medskip

\noindent\emph{Remark.}
\emph{Algebraic Eraser} is a trademark of SecureRF.
The variant of CBKAP actually implemented by SecureRF uses proprietary distributions,
and thus our results do not imply its vulnerability.

\medskip

\noindent Received 10 August 2009; received in revised form 18 May 2011; accepted 5 March 2012.
\end{abstract}

\maketitle

\section{Introduction and overview}

Starting with the seminal papers \cite{Anshel, Ko}, several attempts
have been made to construct and analyze public key schemes based on noncommutative groups
and combinatorial, or computational, group theory.
One motivation is that such systems may provide longer term security than existing schemes.
Another motivation, at present theoretical, is that unlike the main present day public key schemes,
these new schemes may be resistant to attacks by quantum computers.
Already in the short run, these connections between
combinatorial group theory and cryptography lead to mathematical questions not asked
before, and consequently to new mathematical and algorithmic results.

In this paper, we study a scheme falling in the above category, whose cryptanalysis
leads to an algorithm with potential applicability beyond the studied scheme.

The \emph{Algebraic Eraser} key agreement scheme was introduced by Anshel, Anshel, Goldfeld, and Lemieux in
the workshop \emph{Algebraic Methods in Cryptography} held in Dortmund, Germany, on March 2004, and
in the special session on \emph{Algebraic Cryptography}, at the Joint International Meeting of
the AMS, DMV, and \"OMG, held in Mainz, Germany, on June 2005.
It was subsequently published as \cite{Eraser}.

Apart from its mathematical novelty, the Algebraic Eraser
has a surprisingly simple concrete realization, the \emph{Colored Burau Key Agreement Protocol (CBKAP)},
which consists of an efficient combination of matrix multiplications, applications of permutations, and
evaluations of polynomials at elements of a finite field.

For four years since its introduction, no weakness in CBKAP was reported.
On January 13, 2008, Kalka and Tsaban have described the attack presented here in
Bar-Ilan University's \emph{CGC Seminar} \cite{CGCColloq}. At about the same time (on January 30, 2008),
Myasnikov and Ushakov uploaded to the ArXiv eprint server an
independent attack, and subsequently published it \cite{MU}.
Myasnikov and Ushakov use a length based algorithm to break the Third Trusted Party (TTP)
component of CBKAP, with excellent success rates for the parameters proposed in \cite{Eraser}.
They indicate that the success rates of their attack drop if the parameters are increased.
In his recent paper \cite{Gunnells}, Gunnells reproduces Myasnikov and Ushakov's attack,
and concludes that as the key lengths increase, the attack quickly loses power, and soon fails in all instances.
Furthermore, he provides experiments suggesting that their attack is not robust against several easily
implemented defenses.
He concludes that ``the success of the attack seems mainly to be due to it
being applied to short words.'' \cite{Gunnells}

The security of the main ingredient of CBKAP is not addressed in \cite{MU}.
Would fixing the TTP component make the protocol secure?
Moreover, in the recent work \cite{Eraser2} it is shown that
in many scenarios, there is no need to make both groups $A$ and $B$ in the protocol public (details below).
For this variant, the attack presented in \cite{MU} does not seem applicable \cite{Gunnells}.

We present an efficient attack, which recovers the shared key out of the public information,
even if one of the involved groups mentioned above remains hidden,
without attacking the TTP's private key.
According to heuristic reasoning as well as massive experiments,
the attack is efficient, and has 100\% success rates for all feasible sizes of the security parameters,
assuming standard distributions on the key spaces.\footnote{See the remark at the abstract,
which applies to this paper as well as to \cite{MU, Gunnells}.}

The methods, which make the attack applicable to large security parameters,
come from probabilistic group theory, and deal with permutation groups.
About half of the paper is dedicated to a new heuristic algorithm for finding
short expressions of permutations as words in a given set of randomly chosen permutations.
This algorithm solves efficiently instances which are intractable using previously known, provable or
heuristic, techniques.

\medskip

We conclude this introduction with several general comments.

\subsection{Construction versus analysis}
Few schemes, not counting minor variations, have been proposed thus far in the context of combinatorial group theory:
Mainly, those in \cite{Anshel, Ko}, and the one from \cite{Eraser}, which is studied here.
Most of the attempts thus far are on the side of analysis of the proposed schemes, rather than proposals of substantially new ones.
Indeed, each of the mentioned schemes is in fact an infinite family of possible schemes, with at least
two degrees of freedom: Choosing the platform groups, and choosing the distributions on the chosen
platform group. There are at present no known attacks which are guaranteed to succeed against all candidate groups (with standard distributions
on them), or against all distributions on certain groups (like the braid group) which can be sampled efficiently.
Thorough analyzes may give indication which choices may lead to secure schemes.

\subsection{Key generation in infinite groups or monoids}
There is a canonical distribution on infinite groups or monoids generated by finitely many
generators $g_1,\dots,g_k$. This is defined by fixing a length parameter $L$, and then
taking a product of $L$ elements $g_i$, each chosen with uniform distribution from the set $\{g_1,\dots,g_k\}$.\footnote{In the
case of a group, we first extend the list of generators, if necessary, so that for each $g$ in the list, $g\inv$ is also in the list.}
This is not a uniform distribution on the group,\footnote{Since the group or monoid is countably infinite, there is no uniform distribution on it.}
but the distribution induced from the uniform distribution on the
words of length $L$ in the free monoid. Since the distributions are not specified in \cite{Eraser}, we
(as well as Myasnikov and Ushakov \cite{MU} and Gunnells \cite{Gunnells})
assume these natural distributions when finitely generated groups or monoids are considered, and the uniform distribution
when finite sets are considered.
By Gunnells result \cite{Gunnells}, the results of the present paper form the first cryptanalysis
of CBKAP for these distributions, which works for all key sizes.

\subsection{Provable security}
Given a cryptographic scheme, it is desirable to have a simply stated (and apparently hard) algorithmic problem
such that, if the given scheme can be broken, then there is an efficient algorithm for the problem.
From a cryptographic point of view, there is no point in doing so when a scheme can be cryptanalyzed, as is the case here.
We believe, however, that cryptanalyses will improve our understanding of schemes based on
combinatorial group theory, and this may lead, eventually, to introduction of schemes which look promising
(i.e., resist known cryptanalyses).
Then, establishing a provable link between the security of the scheme and
the difficulty of a simply stated and apparently hard algorithmic problem would be an important task,
which would help understanding better the (potential) security of the scheme.

\section{The Algebraic Eraser scheme}
We describe here the general framework. The concrete realization will be described later.

\subsection{Notation, terminology, and conventions}
A \emph{monoid} is a set $M$ with a distinguished element $1\in M$,
equipped with an associative multiplication operation for which $1$ acts
as an identity. Readers not familiar with this notion may replace ``monoid''
with ``group'' everywhere, since this is the main case considered here.

Let $G$ be a group acting on a monoid $M$ on the left,
that is, to each $g\in G$ and each $a\in M$, a unique
element denoted $\act{g}{a}\in M$ is assigned, such that:
\be
\itm $\act{1}{a} = a$;
\itm $\act{gh}{a} = \act{g}{(\act{h}{a})}$; and
\itm $\act{g}{(ab)} = \act{g}{a}\cdot\act{g}{b}$
\ee
for all $a,b\in M, g,h\in G$.

$M\x G$, with the operation
$$\op{a}{g}\circ \op{b}{h} = \op{a\cdot \act{g}{b}}{gh},$$
is a monoid denoted $M \rtimes G$, and called the \emph{semi-direct product} of $M$ and $G$.

Let $N$ be a monoid, and $\vphi:M\to N$ a homomorphism.
The \emph{algebraic eraser} operation is the function
$\star : (N\x G)\x (M \rtimes G)\to (N\x G)$ defined by
\begin{equation}\label{stardef}
\op{a}{g}\star \op{b}{h} = \op{a  \vphi(\act{g}{b})}{gh}.
\end{equation}
$M \rtimes G$ \emph{acts on the right} on $N\x G$, that is, the following identity holds:
\begin{equation}\label{star-circ}
\bigl( (a,g)\star (b,h)\bigr) \star \op{c}{r} = \op{a}{g}\star\bigl((b,h)\circ (c,r)\bigr)
\end{equation}
for all $(a,g)\in N\x G, (b,h),(c,r)\in M\x G$.

Submonoids $A,B$ of $M \rtimes G$ are \emph{$\star$-commuting} if
\begin{equation}\label{starcommute}
\op{\vphi(a)}{g}\star \op{b}{h} = \op{\vphi(b)}{h}\star\op{a}{g}
\end{equation}
for all $(a,g)\in A, (b,h)\in B$.
In particular, if $A,B$ $\star$-commute, then
$$\vphi(a) \vphi(\act{g}{b}) = \vphi(b)  \vphi(\act{h}{a})$$
for all $(a,g)\in A, (b,h)\in B$.

\subsubsection{Didactic convention} Since the actions are superscripted,
we try to minimize the use of subscripts. As a rule, whenever two
parties, Alice and Bob, are involved, we try to use for Bob letters which are subsequent
to the letters used for Alice (as is suggested by their names).

\subsection{The Algebraic Eraser Key Agreement Scheme}

\subsubsection{Public information}
\be
\itm The group $G$ and the monoids $M, N$.
\itm A positive integer $m$.
\itm $\star$-commuting submonoids $A,B$ of $M \rtimes G$,
each given in terms of a generating set of size $k$.
\itm Element-wise commuting submonoids $C, D$ of $N$.
\ee

\brem
For clarity of exposition, we assume that $m,k$ are public, and identical for Alice's and Bob's parts.
However, this assumption is not required for the scheme to work, nor for its cryptanalysis described
below.
\erem

\subsubsection{The protocol}\label{prot1}
\be
\itm Alice chooses $c\in C$ and $(a_1,g_1),\dots,(a_m,g_m)\in A$,
and sends
$$(p,g) = (c,1)\star (a_1,g_1)\star \dots \star(a_m,g_m)\in N\x G$$
(the $\star$-multiplication is carried out from left to right) to Bob.
\itm Bob chooses $d\in D, (b_1,h_1),\dots,(b_m,h_m)\in B$,
and sends
$$(q,h) = (d,1)\star (b_1,h_1)\star \dots \star(b_m,h_m)\in N\x G$$
to Alice.
\itm Alice and Bob compute the shared key:
$$
(c  q,h) \star (a_1,g_1)\star \dots \star(a_m,g_m) =
(d  p,g)\star (b_1,h_1)\star \dots \star(b_m,h_m).
$$
We will soon explain why this equality holds.
\ee
For the sake of mathematical analysis, it is more convenient to
reformulate this protocol as follows. The public information
remains the same.
In the notation of Section \ref{prot1}, define
\begin{eqnarray*}
(a,g) & = & (a_1,g_1)\circ \dots \circ (a_m,g_m)\in A;\\
(b,h) & = & (b_1,h_1)\circ \dots \circ (b_m,h_m)\in B.
\end{eqnarray*}
By Equations \eqref{star-circ} and \eqref{stardef}, Alice and Bob transmit the information
\begin{eqnarray*}
(p,g) & = & (c,1)\star (a_1,g_1)\star \dots \star(a_m,g_m) = (c,1)\star (a,g) = (c  \vphi(a),g);\\
(q,h) & = & (d,1)\star (b_1,h_1)\star \dots \star(b_m,h_m) = (d,1)\star (b,h) = (d  \vphi(b),h).
\end{eqnarray*}
Using this and Equation \eqref{starcommute}, we see in the same manner that the shared key is
\begin{eqnarray}
\lefteqn{(cq,h) \star (a,g) = (c  q\vphi(\act{h}{a}),h  g) =}\\ \nonumber
& = & (c  d\vphi(b)\vphi(\act{h}{a}),hg) = (d c\vphi(a)\vphi(\act{g}{b}),gh) =\\ \nonumber
& = & (d p\vphi(\act{g}{b}),gh) = (d p,g)\star (b,h).\label{shared}
\end{eqnarray}

\subsection{When $M$ is a group}\label{agroup}
In the concrete examples for the Algebraic Eraser scheme, $M$ is a group \cite{Eraser}.
Consequently, $M\rtimes G$ is also a group, with inversion
$$(a,g)\inv = (\act{g\inv}{a\inv},g\inv)$$
for all $(a,g)\in M\rtimes G$.

\section{A general attack on the scheme}

We will attack a stronger scheme, where only one of the groups $A$ or $B$ is made public.
Without loss of generality, we may assume that $A$ is known.
$A$ is generated by a given $k$-element subset. Let
$(a_1,s_1),\dots,(a_k,s_k)\in M\rtimes G$ be the given generators of $A$.
Let $S = \{s_1,\dots,s_k\}$.
$S^{\pm 1}$ denotes the symmetrized generating set $\{s_1,\dots,s_k,\allowbreak s_1\inv,\dots,s_k\inv\}$.

\subsection{Assumptions}

\subsubsection{Distributions and complexity}
Alice and Bob make their choices according to certain distributions.
Whenever we mention a \emph{probability}, it is meant with respect to the relevant distribution.
All assertions made here are meant to hold ``with significant probability'' and the generation
of elements must be possible within the available computational power. We will quantify
our statements later.

\begin{as}\label{loworder}
It is possible to generate an element $(\alpha,1)\in A$ with $\alpha\neq 1$.
\end{as}
Assumption \ref{loworder} is equivalent to the possibility of generating $(\alpha,g)\in A$ such
that the order $o$ of $g$ in $G$ is smaller than the order of $(\alpha,g)$ in $M\rtimes G$. Indeed, in this
case $(\alpha,g)^o$ is as required.

\begin{as}\label{matrices}
$N$ is a subgroup of $\operatorname{GL}_n(\F)$ for some field $\F$ and some $n$.
\end{as}
We do not make any assumption on the field $\F$.

Alice generates an element $(a,g)\in A$, and in particular she generates $g$ in the subgroup of
$G$ generated by $S$.

\begin{as}\label{shortexpression}
Given $g\in\langle S\rangle$, $g$ can be explicitly expressed as a product of elements of $S^{\pm 1}$.
\end{as}

\subsection{The attack}

\subsubsection{First phase: Finding $d$ and $\vphi(b)$ up to a scalar}\label{firstphase}
By $\star$-commutativity of $A$ and $B$, and since $(b,h)\in B$,
we have that for each $(\alpha,1)\in A$,
\begin{equation}\label{1}
\vphi(\alpha)\vphi(b)  =  \vphi(b)\vphi(\act{h}{\alpha}).
\end{equation}
By Assumption \ref{loworder}, we can generate such equations with $\alpha$ known,
so that only $\vphi(b)$ is unknown.

Now, $q = d\vphi(b)$ is a part of the transmitted information.
Substituting $\vphi(b) = d\inv q$ in Equation \eqref{1}, we obtain
$$d\vphi(\alpha)=(q\vphi(\act{h}{\alpha})q\inv) d,$$
where only $d$ is unknown. Moreover, as $C,D$ commute element-wise,
we have that
\begin{equation}\label{2}
d\gamma  =  \gamma d
\end{equation}
for all $\gamma\in C$.

Even for just one nontrivial $\alpha$
and one $\gamma\in C$, we obtain $2n^2$ equations on the $n^2$ entries of $d$.
Thus, if standard distributions were used to generate the keys,
we expect, heuristically, that the solution space will be one-dimensional.
(As this is a homogeneous equation and the matrices are invertible,
the solution space cannot be zero-dimensional.)
If it is accidentally not, we can generate more equations in the same manner.\footnote{In the case of CBKAP (the concrete realization described below),
one equation of each type was enough in all experiments we have conducted.
Except for few exceptions in tiny parameter settings, where exhaustive
search of the key can be carried out easily.}

More formally, let $d,\tilde d$ be solutions to Equations \eqref{1} and \eqref{2},
say for $(\alpha_1,1),\dots,\allowbreak(\alpha_r,1)\in A$, and $\gamma_1,\dots,\gamma_s\in C$.
Then
\begin{eqnarray*}
\tilde dd\inv & = & \tilde d\vphi(\alpha_i)(d\vphi(\alpha_i))\inv =
(q\vphi(\act{h}{\alpha_i})q\inv)\tilde d ((q\vphi(\act{h}{\alpha_i})q\inv) d)\inv =\\
& = & (q\vphi(\act{h}{\alpha_i})q\inv)\tilde d d\inv(q\vphi(\act{h}{\alpha_i})q\inv)\inv,
\end{eqnarray*}
and thus $\tilde dd\inv$ commutes with $q\vphi(\act{h}{\alpha_i})q\inv$,
for each $i\in\{1,\dots,r\}$.

Moreover, $\tilde dd\inv$ commutes with all elements of $C$, and thus
is in the centralizer of
$$\{q\vphi(\act{h}{\alpha_1})q\inv,\dots,q\vphi(\act{h}{\alpha_r})q\inv\}\cup \{\gamma_1,\dots,\gamma_s\}.$$
Similarly, we have that $d\inv\tilde d$ is in the centralizer of
$$\{\vphi(\alpha_1),\dots,\vphi(\alpha_r)\}\cup \{\gamma_1,\dots,\gamma_s\}.$$
Our precise assumption is that for some, not large, numbers $r,s$,
we have that with high probability, (at least) one of these centralizers is one-dimensio\-nal.
Since $C$ is, by assumption, a group of matrices, this means that this centralizer is not larger
than the centralizer of the full matrix group $\operatorname{GL}_n(\F)$,
i.e.\ the scalar matrices. (Observe that $d\inv\tilde d$ is scalar if and only if $\tilde dd\inv$ is.)

If one can find a small generating set $\{\gamma_1,\dots,\gamma_s\}$ for $C$,\footnote{Indeed, in CBKAP,
described below, $C$ is cyclic.}
then the assumption tells that the centralizer of
$$\{\vphi(\act{h}{\alpha_1}),\dots,\vphi(\act{h}{\alpha_r})\}\cup q\inv Cq$$
or of
$$\{\vphi(\alpha_1),\dots,\vphi(\alpha_r)\}\cup C$$
is one-dimensional.

Thus, heuristically, we assume that we have found $x d$ for some unknown scalar $x\in\F$.
Now use our knowledge of $q = d\vphi(b)$ to compute
$$(x d)\inv q = \frac{1}{x}\ d\inv q = \frac{1}{x}\ \vphi(b).$$
In summary:
We know $x d$ and $x\inv\vphi(b)$,  for some unknown scalar $x\in\F$.

\subsubsection{Second phase: Generating elements with a prescribed $G$-coordina\-te and extracting the key}\label{secondphase}
Using Assumption \ref{shortexpression},
find $i_1,\dots,i_\ell\in\{1,\dots,k\}$ and $\epsilon_1,\dots,\epsilon_\ell\in\{1,-1\}$
such that
$$g = s_{i_1}^{\epsilon_1}\cdots s_{i_\ell}^{\epsilon_\ell}.$$
Compute
$$(\delta,g) = (a_{i_1},s_{i_1})^{\epsilon_1}\circ \cdots \circ (a_{i_d},s_{i_\ell})^{\epsilon_\ell}\in A.$$
$\delta$ may or may not be equal to $a$.

\begin{rem}
If $M$ is finitely generated as a monoid, the expression in Assumption \ref{shortexpression} should
be as a product of elements of $S$.
In the cases discussed later in this paper, $G=S_n$ and the methods of
Section \ref{heualg} can be adjusted to obtain positive expressions (Remark \ref{mono}).
\end{rem}

By $\star$-commutativity of $(\delta,g)$ and $(b,h)$,
$\vphi(b)\vphi(\act{h}{\delta}) = \vphi(\delta)\vphi(\act{g}{b})$,
and thus we can compute
$$x\inv\vphi(\act{g}{b}) = \vphi(\delta)\inv (x\inv \vphi(b))\vphi(\act{h}{\delta}).$$
We are now in a position to compute the secret part of the shared key, using Equation \eqref{shared}:
$$(xd)p(x\inv\vphi(\act{g}{b})) =dp\vphi(\act{g}{b}).$$
The attack is complete.

\section{Cryptanalysis of CBKAP}

Anshel, Anshel, Goldfeld, and Lemieux propose in \cite{Eraser}
an efficient concrete realization which they name \emph{Colored Burau Key Agreement Protocol (CBKAP)}.
We give the details, and then describe how our cryptanalysis applies in this case.

\subsection{CBKAP}\label{CBKAPdef}
CBKAP is the Eraser Key Agreement scheme in the following particular case. Fix positive integers $n$ and $r$, and a prime number $p$.
\be
\itm $G = S_n$, the symmetric group on the $n$ symbols $\{1,\dots,n\}$.
$S_n$ acts on $M=\operatorname{GL}_n(\F_p(t_1,\dots,t_n))$ by permuting the variables $\{t_1,\dots,t_n\}$.
\itm $N = \operatorname{GL}_n(\F_p)$.
\itm $M\rtimes S_n$ is the subgroup of $\operatorname{GL}_n(\F_p(t_1,\dots,t_n))\rtimes S_n$, generated by
$(x_1,s_1),\allowbreak\dots,(x_{n-1},s_{n-1})$, where $s_i$ is the transposition $(i,i+1)$, and\\
$x_1 = \mx{-t_1 & 1\\
0 & 1\\
 & & \ddots\\
 & & & 1};\quad
x_i = \mx{1\\
& \ddots\\
& & 1 & 0 & 0\\
& & t_i & -t_i & 1\\
& & 0 & 0 & 1\\
& & & & & \ddots\\
& & & & & & 1
}
$\\
for $i=2,\dots,n-1$.
Only the $i$th row of $x_i$ differs from the corresponding row of the identity matrix.
A direct calculation shows that $(x_i,s_i)$ commutes with $(x_j,s_j)$ when $|i-j|>1$,
and that
$$(x_i,s_i)(x_{i+1},s_{i+1})(x_i,s_i)=(x_{i+1},s_{i+1})(x_i,s_i)(x_{i+1},s_{i+1}).$$
Thus, the \emph{colored Burau group} $M\rtimes S_n$ is a representation of Artin's braid group $B_n$,
determined by mapping each Artin generator $\sigma_i$ to $(x_i,s_i)$, $i=1,\dots,n-1$.\footnote{Additional
details on the colored Burau group can be found in, e.g., \cite{Morton}.}

\itm $\vphi:M\to\operatorname{GL}_n(\F_p)$ is the evaluation map obtained by replacing each variable $t_i$
by a fixed element $\tau_i\in\F_p$.
\itm $C=D=\F_p(\kappa)$ is the group of nonzero matrices of the form
$$\ell_1\kappa^{j_1}+\dots+\ell_r\kappa^{j_r},$$
with $\kappa\in\operatorname{GL}_n(\F_p)$
a matrix of order $p^n-1$, $\ell_1,\dots,\ell_r\in\F_p$, and $j_1,\dots,j_r\in\mathbb{Z}$.
\ee
\brem
Let $f(x)$ be the characteristic polynomial of $\kappa$.
Then $x\mapsto\kappa$ induces an isomorphism from $\F_{p^n}=\F[x]/\langle f(x)\rangle$
onto $C=D$, that is, $C$ and $D$ are the image of $\F_{p^n}^*$ in $\operatorname{GL}_n(\F_p)$,
or in other words, the nonsplit torus in $\operatorname{GL}_n(\F_p)$.
\erem
Commuting subgroups of $M\rtimes G$ are chosen once, by a trusted party,
as follows:
\be
\itm Fix $I_1,I_2\sbst\{1,\dots,n-1\}$ such that for all $i\in I_1$ and $j\in I_2$,
$|i-j|\ge 2$. $|I_1|$ and $|I_2|$ are both $\le n/2$.
\itm Define $L=\langle \sigma_i : i\in I_1\rangle $ and $U=\langle \sigma_j : j\in I_2\rangle $, subgroups of $B_n$
generated by Artin generators.
\itm $L$ and $U$ commute element-wise. Add to both groups the central element $\Delta^2$ of $B_n$.
\itm Choose a random $z\in B_n$.
\itm Choose $w_1,\dots,w_k\in zLz\inv, v_1,\dots,v_k\in zUz\inv$, each a product of $t$
generators ($t$ is a parameter of the scheme). Transform them into Garside left normal form, and remove all even powers of $\Delta$.
Reuse the names $w_1,\dots,w_k, v_1,\dots,v_k$ for the resulting braids.
\itm Let $\rho: B_n\to M\rtimes S_n$ be the colored Burau representation function.
$A,B$ are the subgroups of $\rho(zLz\inv),\rho(zUz\inv)$ generated by $\rho(w_1),\dots,\rho(w_k)$,
and by $\rho(v_1),\dots,\rho(v_k)$, respectively.
\itm $w_1,\dots,w_k, v_1,\dots,v_k$ are made public.
\ee
To carry out our attack, it suffices to assume that the image in the colored Burau group
of one of the sets $\{w_1,\dots,w_k\}$ or $\{v_1,\dots,v_k\}$, is given.

\subsubsection{Parameter settings and efficiency}
These issues are discussed in detail in \cite{Eraser}.
For the parameters proposed there, it is shown that CBKAP can be implemented efficiently, even on small devices
as RFID tags. However, we are interested in the more general question, whether \emph{some} parameters
may make CBKAP secure. E.g., CBKAP can be implemented on standard PC-s with parameters much larger than
those proposed in \cite{Eraser}, and still be more efficient than the ordinary schemes based on RSA,
Diffie-Hellman in $\mathbb{Z}_p^*$, or elliptic curves.
We will show that even for such large parameters, CBKAP can be broken.

\subsection{The attack}
Assumption \ref{matrices}, that $N$ is a subgroup of $\operatorname{GL}_n(\F)$ for some field $\F$,
is a part of the definition of CBKAP. We consider the remaining ones.

\subsubsection{Regarding Assumption \ref{loworder}}
This assumption amounted to:
It is possible to generate, efficiently, an element $(\alpha,\sigma)\in A$ such
that the order $o$ of $\sigma$ is smaller than that of $(\alpha,\sigma)$.

In the notation of Section \ref{CBKAPdef},
$\{i,i+1 : i\in I_1\}$ decomposes to a family $\cI$ of maximal intervals
$[i,\ell] = \{i,i+1,\dots,\ell\}$, and $\sum_{[i,\ell]\in\cI}\ell-i+1\le n/2$.
Now
$$U = \langle \Delta^2\rangle \oplus\bigoplus_{[i,\ell]\in I} B_{\ell-i+1}.$$
Each considered $s$ is a permutation induced by the braid $\Delta^{2m}zwz\inv$ with $w\in L$.
Let $\pi:B_n\to S_n$ be the canonical homomorphism. Then
$$s = \pi(\Delta^{2m}zwz\inv) = \pi(\Delta^2)^m\pi(z)\pi(w)\pi(z)\inv = \pi(z)\pi(w)\pi(z)\inv,$$
is conjugate to $\pi(w)$. On each component, this is a product of many random transpositions,
and is therefore an almost uniformly-random permutation on that component.
We therefore have the following:
\be
\itm $U/\langle \Delta^2\rangle $ decomposes into a direct product of braid groups, whose indices
do not sum up to more than $n/2$.
\itm $\pi(U)$ decomposes into a direct product of symmetric groups, whose indices
do not sum up to more than $n/2$.
\itm For generic\footnote{By ``generic'' we mean a typical element with respect to the relevant distribution.}
$(a,s)\in A$, $\pi(z)\inv s\pi(z)$ is generic on each part of the mentioned decomposition.
\ee
The probability that the order of a
random permutation in $S_n$ is $\le n$ is $O(1/\sqrt[4]{n})$ \cite{BabaiHayes05}.
Thus, we can find an element $(a,s)\in A$ with $s$ of order $\le n$ by
generating (roughly $\sqrt[4]{n}$) elements $(a,s)\in A$, until the order of $s$ is as required.

On the other hand, the element $(a,s)$ is a representation of an element of the braid group,
which is known to be torsion-free \cite{Mu82}. While it may be that the representation used here is
not faithful,\footnote{It is open whether the colored Burau representation is faithful,
even without reduction of the integers modulo $p$.}
it is very unlikely that $(a,s)$ could have finite order.

\medskip

The remainder of this paper is dedicated to Assumption \ref{shortexpression}.

\section{Membership search in generic permutation groups}\label{heualg}


For the second phase of our attack, we need to find a short
expression of a given permutation from $G$ in terms of given ``random'' permutations.
In CBKAP, the group $G$ typically has the form $\pi\inv H\allowbreak\pi\le S_n$, where $\pi\in S_n$,
$H$ is $S_{n/2}$ or $A_{n/2}$, and $H$ is embedded in $S_n$ in a natural way (supported by the $n/2$ higher indices).
The conjugation is just relabeling of the indices $1,\dots,n$.
Thus, we may reduce the problem to the case $G=S_n$.
Modifications of the algorithm can be made,
that will make it applicable to any (conjugation of) direct product of groups of the form
$A_n$ or $S_n$.

For \emph{concrete} generators, the problem of finding short expressions for given permutations
is well known, and in similar form occurs in the analysis of the Rubik's cube and other
puzzles. The best known heuristics for solving it in these cases
are based on Minkwitz's algorithms \cite{Min98}, and are incapable of
managing Problem \ref{randsp} for random $s_1,\ldots ,s_k\in S_n$ where
$n$ is large (say, $n\ge 128$), as our experiments below show.

\begin{prob}\label{randsp}
Given random $s_1,\ldots ,s_k\in S_n$ and $s\in \langle s_1,\ldots ,s_k\rangle $,
express $s$ as a short product of elements from $\{s_1,\ldots ,s_k\}^{\pm 1}$.
\end{prob}

In Problem \ref{randsp}, \emph{short} could mean of polynomial length, or of length manageable
by the given computational power as explained above. In any case, the length is the number of letters in the expression,
and not the length of a compressed version of the expression. This limitation comes from the intended
application, where elements of the infinite monoid require storage space which grows with multiplication,
and circumventing this problem by performing one $\star$ multiplication for each letter in the word makes it impossible
to square in a single operation.
If the word is too long (e.g., of the form $a^{(2^{64})}$ for a single generator $a$), the second phase of the attack
becomes infeasible.

Much work was carried out on this problem, by Babai, Beals, Hetyei, Hayes, Kantor, Lubotzky, Seress,
and others (see \cite{BHKL90, BabiBealsSeress04, BabaiHayes05} and references therein).
The works of Babai, Beals, Hayes, and Seress \cite{BabaiHayes05, BabiBealsSeress04} imply that
there is a Las Vegas algorithm for Problem \ref{randsp}, producing expressions of length $n^7(\log n)^{O(1)}$.
This remarkable result solves our problem for moderately small values of $n$.
However, for $n\ge 128$, the resulting expression is too long to be practical.
Our algorithm may be viewed as a (substantial) heuristic simplification of the algorithms induced by
their works.

A classical result of Dixon \cite{Dixon69} tells that two random elements of $S_n$,
almost always generates $A_n$ (if all generators are even permutations) or $S_n$ (otherwise).
Babai proved that getting $A_n$ or $S_n$ happens in probability $1-1/n+O(1/n^2)$ \cite{Babai89}.
Moreover, experiments show that this probability is very close to $1-1/n$ even for small $n$,
i.e., the $O(1/n^2)$ is negligible also for small $n$.
This generalizes to arbitrary $k$, as follows.

\begin{thm}[Dixon]\label{BD}
The asymptotic probability that $k$ random elements of $S_n$ generate $A_n$ or $S_n$ is
roughly $1-n^{-k+1}$.
\end{thm}
Since we do not know of a reference for a proof, we include a proof, suggested to us by Dixon.
\bpf[Proof of Theorem \ref{BD}]
In Section 4 of \cite{Dixon05} it is shown that the probability that
$k$ random permutations generate a transitive group is roughly $1-n^{-k+1}$.

The proof of Lemma 2 in \cite{Dixon69}
can be modified to show that the proportion of pairs $(x,y)$ with $x,y\in S_n$ which are contained
in an imprimitive group (not necessarily generating the imprimitive group) is at most $2^{-n/4}$.
Hence, the proportion of $k$-tuples contained in an imprimitive group is bounded by
$2^{-n/4}$ for all $k>1$.

Theorem 2.8 of Babai's paper \cite{Babai89}
states that the probability that $k$ random permutations generate a primitive group different from $A_n$ or $S_n$ is
smaller than $(n^{\sqrt{n}}/n!)^{k-1}$ (generalizing his theorem when $k = 2$), which is exponentially small.
\epf

Given that we obtain $A_n$ or $S_n$, the probability of the former case is
$2^{-k}$. However, since $k=2$ is of classical interest, we do not
neglect this case.
Thus, for randomly chosen permutations Problem \ref{randsp} reduces (with a small loss in probability)
to the following one.

\begin{prob}\label{randsp2}
\mbox{}
\be
\item Given random $s,s_1,\ldots ,s_k\in A_n$,
express $s$ as a short product of elements from $\{s_1,\ldots ,s_k\}^{\pm 1}$.
\item Given random $s,s_1,\ldots ,s_k\in S_n$ with some $s_i\notin A_n$,
express $s$ as a short product of elements from $\{s_1,\ldots ,s_k\}^{\pm 1}$.
\ee
\end{prob}

A solution of Problem \ref{randsp2}(1) implies a solution of Problem \ref{randsp2}(2):
Let $I=\{i : s_i\notin A_n\}$. $I\neq\emptyset$. Fix $i_0\in I$, and for each $i\in I$,
replace the generator $s_i$ with the generator $s_{i_0}s_i\in A_n$.
Then $\{s_{i_0}s_i : i\in I\}\cup\{s_i : i\notin I\}$ is a set of $k$
nearly random elements of $A_n$ (cf.\ \cite{BabaiHayes05}).
If $s\in A_n$, use (1) to obtain a short expression of $s$ in terms of the new generators.
This gives an expression in the original generators of at most double length.
Otherwise, $s_{i_0}s\in A_n$ and its expression gives an expression of $s$ in terms
of the original generators.

Thus, in principle one may restrict attention to Problem \ref{randsp2}(1).
However, we do not take this approach,
since we want to make use of transpositions when we can.

\subsection{The algorithm} \label{Alg}

\subsubsection{Conventions}
\mbox{}
\be
\itm During the algorithm's execution, the expressions of some of the computed permutations in terms of
the original generators should be stored.
We do not write this explicitly.
\itm The statement \emph{for each $\tau\in\langle S\rangle $} means that the
elements of $\langle S\rangle $ are considered one at a time, by first considering
the elements of $S^{\pm 1}$, then all (free-reduced) products of two elements from $S^{\pm 1}$,
etc.\ (a breadth-first search), until an \emph{end} statement is encountered.
\itm For $s\in (S^{\pm 1})^*$, $\len(s)$ denotes the length of $s$ as a free-reduced word.
$s$ is identified in the usual way with the permutation which is the product of the letters in $s$.
\ee
We are now ready to describe the steps of our algorithm.
We do not consider the question of optimal values for the
parameters and other optimizations. This is left for future investigation.

\newcommand{\mypar}[2]{\par\smallskip\noindent{\textbf{#1:}} \emph{#2}}

\mypar{Input}{} $G=S_n$ or $A_n$; generators $s_1,\ldots,s_k$ of $G$; $s\in G$.

\mypar{Initialization}{}
$$c=\begin{cases}
2 & G=S_n\\
3 & G=A_n
\end{cases}$$
$C$ is the set of $c$-cycles in a canonical expression of $s$ as a product of $c$-cycles.

\mypar{Step 1}{Find a short $c$-cycle in $\langle s_1,\dots,s_k\rangle $.}
\begin{tabbing}
For\={} each $\tau\in\langle s_1,\dots,s_k\rangle $:\\
\> If\={} there is $m\in\{1,\dots,n\}$ such that $\tau^m$ is a $c$-cycle:\\
\> \> $\mu \leftarrow \tau^m$;\\
\> \> End Step 1.
\end{tabbing}
The result $\mu$ of Step 1 is forwarded to the next step.

\mypar{Step 2}{Find short expressions for additional $c$-cycles.}

\begin{tabbing}
$A_0\leftarrow \{\mu\}$;\\
For\={} $l=1,2,\dots$:\\
\> $A_l\leftarrow\emptyset$;\\
\> For\={} each $i\in\{1,\dots,k\}$, each $\epsilon\in\{-1,1\}$, and each $a\in A_{l-1}$:\\
\>\> If $s_i^{-\epsilon} a s_i^{\epsilon}\notin A_0\cup\dots\cup A_l$, add $s_i^{-\epsilon} a s_i^{\epsilon}$ to $A_l$;\\
\> When\={} $C\sbst A_0\cup\dots\cup A_l$:\\
\> \> End Step 2.
\end{tabbing}

\mypar{Final step}{Find a short expression for $s$.}
\par\smallskip\noindent
Use the expressions of the $c$-cycles in $C$ to get an expression of $s$ in terms of the original generators.

\begin{rem}[Positive expressions]\label{mono}
If one seeks for a \emph{positive} expression for $s$ in terms of $\{s_1,\dots,\allowbreak s_k\}$, we can repeatedly activate Step 1,
consider only words $\tau\in S^*$, to generate enough $c$-cycles to present $s$.
This algorithm is more time consuming this way.
\end{rem}

\section{Analysis of the generic membership search algorithm}\label{ideal}

\subsection{Asymptotic analysis}\label{Aa}
We provide an asymptotic (in $n$) analysis of the time complexity and the final expression length,
for the generic membership search algorithm (Section \ref{Alg}), modulo a probabilistic conjecture,
which we later support by heuristic reasoning as well as extensive experiments.

\subsubsection{Step 1}
\begin{conj}[Minimal Cycle Conjecture]\label{MCC}
Let $S$ be a set of $k$ elements of $S_n$, each chosen independently, according to the uniform distribution on $S_n$.
Let $c=3$ if all elements of $S$ are even, and $2$ otherwise.

Consider the following list: The elements of $S^{\pm 1}$, followed by all products of two of elements of $S^{\pm 1}$,
followed by all products of three elements of $S^{\pm 1}$, etc.

Then, almost always,\footnote{That is, with probability approaching $1$ as $n\to\infty$.}
there is among the first $n^2$ elements of the list
an element $\tau$ such that $\tau^m$ is a $c$-cycle, for some $m\le n$.
\end{conj}

In short, the Minimal Cycle Conjecture asserts that in Step 1, almost always, at most $n^2$ permutations are considered.

\begin{lem}\label{bt}
Given an element $\tau\in S_n$, deciding whether there is $m\le n$ such that $\tau^m$ is a $c$-cycle
can be done in time $O(n)$.
Finding the minimal such $m$, if it exists, can be done in time $O(n\log n)$.
\end{lem}
\begin{proof}
Let $\ell_1,\ell_2,\dots,\ell_k$ be the cycle lengths in the cycle decomposition of $\tau$, which can be computed
in linear time.
There is $m$ as required if, and only if,
there is a unique $i\le k$ such that $\ell_i=c$, and for each $j\le k$ different from $i$, $\ell_j$ is relatively
prime to $c$, which can be verified in time $O(\log\ell_j)$.\footnote{Considering standard CPU operations as requiring constant time.}
As $\sum_{j\neq i}\log\ell_j\le \sum_j\ell_j\le n$, the whole procedure is $O(n)$.

The minimal power $m$, if it exists, is $\lcm(\ell_j : j\neq i)$.
This can be computed by partitioning the list into pairs, computing the $\lcm$ of each pair to obtain
a half length list, and doing the same for this list. After $O(\log k)$ steps, we obtain the $\lcm$.
Each step requires roughly $\sum\log\ell_i$ operations, and overall we have $(\log k)\sum\log\ell_i$
which is $O(n\log n)$.
\end{proof}

\begin{cor}\label{LenMu}
Assume that Minimal Cycle Conjecture.
Then, almost always, Step 1 produces a $c$-cycle $\mu$, expressed as a product of $O(n\log n)$ elements of $S^{\pm 1}$,
in time and space $O(n^3)$ (or, alternatively, in time $O(n^3\log n)$ and space $O(n)$).
\end{cor}
\begin{proof}
To produce the list of $n^2$ permutations, we have to compute $n^2$ products of permutations (the product of an already computed permutation and
an element of $S^{\pm1}$), each requiring $n$ operations.
After each such multiplication, we check whether the result $\tau$ has the property that $\tau^m$ is a $c$-cycle for some $m\le n$.
This is done by Lemma \ref{bt}.

In summary, we have $n^2$ steps, each consisting one multiplication of permutations, and one linear time decision.
Thus, the overall time complexity is $O(n^3)$.

If $\tau$ is among the first $n^2$ elements of the sequence, then $\tau$ is a product of at most $\log_{2k}(n^2)=O(\log n)$ elements
of $S^{\pm 1}$, and therefore $\tau^m$ is expressed as a product of at most $O(m\log n)$, which is $O(n\log n)$.

Alternatively, one can compute, for each new word in the generators, the whole product.
This increases the time complexity to $O(n^3\log n)$, but reduces the space complexity to
$O(n)$.
\end{proof}

\subsubsection{Step 2}
Let $S$ be a set of $k$ elements of $S_n$, each chosen independently, according to the uniform distribution on $S_n$.
Consider the graph $G_{n,k,c}$ with vertices all $c$-cycles, such that there is an edge
between $u,v$ if and only if there is $r\in S^{\pm 1}$
with $r\inv ur=v$. This graph has $n!/(n-c)!c$ vertices and is $2k$-regular.
In the worst case we have to compute in Step 2 all vertices of this graph before this procedure terminates.
For every $a\in A_l$, $l\ge 1$, keeping track of its predecessor in $A_{l-1}$,
Step 2 computes a spanning tree of this graph, rooted at $\mu$.
Let $\ell$ be the value of $l$ at the termination of Step 2.
$\ell $ is the height of our tree, and the diameter $d$ of this graph satisfies $\ell\le d \le 2\ell$.

For each $a\in A_l$, $1\le l\le \ell -1$, $2k-1$ conjugations are performed. Indeed for each $a\in A_l$, by considering
which conjugator led to it, the inverse conjugator will not lead to anything new and is thus not
performed. Only for the root we perform $2k$ conjugations, but no one for all $a\in A_{\ell}$.
Thus, the overall number of conjugations in this step is bounded above by
$$ 2k|A_0|+(2k-1)\sum_{l=1}^{\ell -1}|A_l|=1+(2k-1)\sum _{l=0}^{\ell}|A_l|-(2k-1)|A_{\ell}|\le\frac{(2k-1)n!}{(n-c)!c}-2k+2. $$

\brem\label{faster}
In fact, as $S$ generates $G$, it also generates it as a monoid, and
thus it suffices to consider conjugations by positive generators only,
so that the overall number of conjugations is less than $(k-1)n!/(n-c)!c -k+2$.
Moreover, with high probability one can restrict attention to just two generators generating
$G$, so the number of conjugations becomes less than $n!/(n-c)!c<n^c/c$.
Here, we have to consider a digraph rather than a graph. The vertices are again all $c$-cycles and there is a directed edge from $u$ to $v$ iff
$r\inv ur=v$ for some $r\in S$. However, the diameter of this digraph is greater or equal than that of $G_{n,k,c}$.
\erem

\begin{cor}\label{Os2}
The running time of Step 2 is $O(n^2)$ if $G=S_n$ and $O(n^3)$ if $G=A_n$.
\end{cor}
\bpf
Let $c=2$ if $G=S_n$ and $3$ if $G=A_n$.
Each conjugation of a $c$-cycle requires $c\le 3$ operations:
\begin{eqnarray*}
p(i\ j)p\inv & = & (p(i)\ p(j))\\
p(i\ j\ k)p\inv & = & (p(i)\ p(j)\ p(k)).
\end{eqnarray*}
Thus, the running time is a small constant times the number of conjugations,
which is bounded by $(2k-1)n^c/c$ (or less if we work according to Remark \ref{faster}).
\epf

For each permutation $\sigma$ encountered during our algorithm,
let $\len(\sigma)$ be the length of its expression as a product of the given permutations $s_1,\dots,s_k$ and their inverses.
Recall that $\mu$ is the output of Step 1. Then for each $c$-cycle $\sigma\in A_0\cup\dots\cup A_\ell$,
$$\len(\sigma)\le \len(\mu)+2\ell.$$

\bco\label{hi}
Using the above notation, the length of the obtained expression for $s$ is
smaller than $n/(c-1)\cdot(\len(\mu)+2\ell)$.
\eco
\bpf
$s$ is a product of at most $n/(c-1)$ $c$-cycles.
\epf

The following theorem consists of Theorems 2.2 and 3.3 of \cite{Friedman}.

\begin{thm}\label{ThFr}
\mbox{}
\be
\itm Fix $k\ge 2$, $c\ge 1$ and a real $\epsilon>0$.

Let $S$ be a set of $k$ elements of $S_n$, each chosen independently, according to the uniform distribution on $S_n$.

Let $D_{n,k,c}$ be the digraph with whose vertices are the $c$-tuples of distinct elements of $\{1,\ldots ,n\}$, and where there is an arrow from
$(a_1,\ldots ,a_c)$ to $(b_1,\ldots ,b_c)$ if and only if $(b_1,\ldots ,b_c)=(s(a_1),\ldots ,s(a_c))$
for some $s\in S$.

Then, almost always, $D_{n,k,c}$ is an $\alpha $-expander, for
$$\alpha =((1-\epsilon)/2)(1-(\sqrt{2k-1}/k)^{1/(1+c)}).$$

\itm The diameter of an $\alpha$-expander with $v$ vertices is smaller than $2(1+\log _{1+\alpha }v)$.
\ee
\end{thm}

\bco \label{logdiam}
For $k\ge 2$, $c\ge 1$, the diameter of the graph $G_{n,k,c}$ is almost always bounded by $2(1+c\log _{1+\alpha}(n))$
with $\alpha =(1-(\sqrt{2k-1}/k)^{1/(1+c)})/2$.
\eco
\bpf
Consider the equivalence relation $\sim$ on the set of $c$-tuples of distinct elements of $\{1,\ldots ,n\}$,
which identifies tuples if each is a cyclic rotation of the other.
The quotient digraph $D_{n,k,c}/\sim $ is exactly the digraph mentioned in Remark \ref{faster}.
Thus, $\ell$ is smaller than its diameter, which is smaller than the diameter of $D_{n,k,c}$.
By Theorem \ref{ThFr}, the latter is smaller than
$$2(1+\log _{1+\alpha}(n^c))=2(1+c\log _{1+\alpha}n).\qedhere$$
\epf

\bco
Assume the Minimal Cycle Conjecture.
The length of the obtained expression for $s$ is $O(n^2\log n)$.
\eco
\bpf
By Corollary \ref{hi}, the length of the obtained expression for $s$ is $O(n(\len(\mu)+2\ell))$.
By Corollary \ref{LenMu}, $\len(\mu)$ is $O(n\log n)$.
By Corollary \ref{logdiam}, the diameter of our graph $G_{n,k,c}$ is $O(\log n)$, and in particular so is $\ell$.
\epf

\subsection{Heuristic evidence and estimation} \label{Heu}

\subsubsection{Step 1}
The following terminology and lemma will make the proof of the subsequent theorem
shorter. The \emph{cycle structure} of a permutation $s\in S_n$ is
the sequence $(n_1,n_2,\dots)$ of lengths of cycles of $s$
which are not fixed points.
Let $\sigma^n_{(n_1,\dots,n_k)}$ denote the number of elements of $S_n$
with cycle structure $(n_1,\dots,n_k)$.

\begin{lem}\label{ss}
For distinct $n_1,\dots,n_k$:
$\sigma^n_{(n_1,\dots,n_k)} = \frac{n!}{(n-(n_1+\dots+n_k))!\cdot n_1\cdots n_k}$.
\end{lem}
\bpf
First choose the $n_1+\dots+n_k$ elements which will occupy the cycles and consider
all their permutations, and
then divide out cyclic rotation equivalence, to get
$$\binom{n}{n_1+\dots+n_k}\cdot(n_1+\dots+n_k)!\cdot\frac{1}{n_1\cdots n_k}.$$
This is equal to $\sigma^n_{(n_1,\dots,n_k)}$.
\epf

\begin{prop}\label{cprob}
Let $c$ be $2$ if $G=S_n$, and $3$ if $G=A_n$. For random $\tau\in G$, the probability that
there is $d\in\{1,\dots,n\}$ such that $\tau^d$ is a $c$-cycle is greater than $1/cn$.
\end{prop}
\bpf
In fact, we give better bounds for most values of $n$.
We consider the probabilities to have cycle structures $(n-d,c)$ or $(n-d,e,c)$ for
appropriate $d$, such that if $\tau$ has such a cycle structure, then $\tau^{n-d}$ is a $c$-cycle.
The restrictions on the cycle structures are as follows.
\be
\itm $c$ does not divide $n-d$; and
\itm $e$ divides $n-d$ (in the case $(n-d,e,c)$).
\ee
In the case $G=A_n$, we also must have that the cycle structure is possible in $A_n$:
\be
\itm[(3)] $n-d$ is odd (in the case $(n-d,3)$);
\itm[(4)] $n-d+e$ is even (in the case $(n-d,e,3)$).
\ee
Assuming these restrictions, we compute the probabilities of these cycle structures
using Lemma \ref{ss}.
In $S_n$, the probability for $(n-d,2)$ is
\begin{eqnarray*}
\frac1{|S_n|}\cdot\sigma^n_{(n-d,2)} & = & \frac{1}{(d-2)!\cdot(n-d)\cdot 2}>\frac{1}{(d-2)!\cdot 2n}.
\end{eqnarray*}
In $A_n$, the probabilities for $(n-d,3)$ and $(n-d,e,3)$ are
\begin{eqnarray*}
\frac1{|A_n|}\cdot\sigma^n_{(n-d,3)} & = & \frac{2}{(d-3)!\cdot(n-d)\cdot 3}>\frac{2}{(d-3)!\cdot 3n},\\
\frac1{|A_n|}\cdot\sigma^n_{(n-d,e,3)} & = & \frac{2}{(d-e-3)!\cdot(n-d)\cdot e\cdot 3}>\frac{2}{(d-e-3)!\cdot 3en},
\end{eqnarray*}
respectively.
We now consider some possible cycle structure with at most one cycle of each length, and describe the restrictions
they pose on $n$ and their probabilities.

For $G=S_n$, we have the following.

\medskip
\begin{center}
\begin{tabular}{|c|l|c|c|c|}
\hline
$n \bmod 2$ & Cycle structure & Prob. & Accumulated probability\\
\hline
\hline
 $0$ & $(n-3,2)$ & $1/2n$ & \\
\cline{2-3}
     & $(n-5,2)$ & $1/12n$ & $7/12n$\\
\cline{1-4}
 $1$ & $(n-2,2)$ & $1/2n$ & \\
\cline{2-3}
     & $(n-4,2)$ & $1/4n$ & $3/4n$\\
\hline
\end{tabular}
\end{center}

\par\medskip
For $G=A_7$, we can compute directly that the cycle structure $(2,2,3)$ has probability $1/12$, which is
greater than $1/3\cdot 7$, as required. For all other $n$, we have the following.

\medskip
\begin{center}
\begin{tabular}{|c|c|l|c|c|c|}
\hline
 $n \bmod 6$ & Cycle structure & Prob. & Accumulated probability\\
\hline
\hline
 0 & $(n-5,3)$ & $1/3n$ & $1/3n$ \\
\cline{1-4}
 1 & $(n-5,2,3)$ & $1/3n$ & \\
\cline{2-3}
   & $(n-6,3)$   & $1/9n$ & $4/9n$\\
\cline{1-4}
 2 & $(n-3,3)$ & $2/3n$ & $2/3n$\\
\cline{1-4}
3 & $(n-4,3)$ & $2/3n$ & \\
\cline{2-3}
 & $(n-5,2,3)$ & $1/3n$ & $1/n$\\
\cline{1-4}
4 & $(n-3,3)$ & $2/3n$ & \\
\cline{2-3}
  & $(n-5,3)$ & $1/3n$ & \\
\cline{2-3}
  & $(n-6,2,3)$ & $1/3n$ & $4/3n$\\
\cline{1-4}
5  & $(n-4,3)$ & $2/3n$ & \\
\cline{2-3}
  & $(n-6,3)$   & $1/9n$ & \\
\cline{2-3}
  & $(n-7,2,3)$ & $1/6n$ & $17/18n$\\
\hline
\end{tabular}
\end{center}

\medskip\noindent
This completes the proof.
\epf

\begin{cor}\label{exp}
Let $c$ be $2$ if $G=S_n$, and $3$ if $G=A_n$.
Execute Step 1 with random elements $\tau\in G$ instead of the enumerated ones.
The probability that it does not end before considering $\lambda n$ permutations is smaller than $e^{-\lambda/c}$.
\end{cor}
\bpf
By Corollary \ref{cprob},
the probability of not obtaining a $c$-cycle for $\lambda n$ randomly chosen
$\tau \in G$ is at most
$$\left(1-\frac{1}{cn}\right)^{\lambda n}=\left(\left(1-\frac{1}{cn}\right)^{cn}\right)^{\frac{\lambda}{c}}<
(e^{-1})^\frac{\lambda}{c} = e^{-\frac{\lambda}{c}}.\qedhere$$
\epf

\begin{exa}
Let $c$ be $2$ if $G=S_n$, and $3$ if $G=A_n$,
and $\lambda=c\lambda_0\log n$ for some constant $\lambda_0$.
Then the probability in Proposition \ref{exp} is smaller than
$$e^{-\frac{c\lambda_0\log n}{c}} = n^{-\lambda_0}.$$
\end{exa}

This shows that if we use, in Step 1 of our algorithm, random elements instead of the enumerated ones,
then, almost always, this step halts after the consideration of at most $3n\log n$ permutations.
This is much smaller than the $n^2$ in the Minimal Cycle Conjecture \ref{MCC}. We conjecture that this
increase from $n\log n$ to $n^2$ remedies for the fact that in the conjecture, the considered elements
are \emph{not} independent.
Below, we provide experimental evidence for that.

\subsubsection{The expression's length}\label{explen}

Using Corollary \ref{logdiam},
we can derive a rough upper bound on the \emph{average} length of the expression
provided by the generic membership search algorithm,
assuming the Minimal Cycle Conjecture \ref{MCC}.
According to this conjecture, Step 1 uses on average less than $n^2$ permutations until finding a good one $\tau$.
If $\tau$ is the $n^2$-th permutation in our breadth-first enumeration of $\langle s_1,\dots,s_k\rangle$,
then its length $d$ as a word in the generators satisfies
$$(2k-1)^{d-1}\le 2k(2k-1)^{d-2}\le n^2.$$
Thus
$$\len(\tau) \lesssim \frac{2}{\log(2k-1)}\cdot\log n.$$
Then, $\mu$ is at most an $n$-th power of $\tau$. Thus on average,
$$\len(\mu) \lesssim \frac{2}{\log(2k-1)}\cdot n\log n.$$
By Corollary \ref{logdiam}, $\ell$ is on average much smaller than $\len(\mu)$,
and thus by Corollary \ref{hi}, the average length of the resulting expression is roughly bounded by
$$\frac{2}{(c-1)\log(2k-1)}\cdot n^2\log n=
\begin{cases}
\frac{2}{\log(2k-1)}\cdot n^2\log n & G=S_n\\
\frac{1}{\log(2k-1)}\cdot n^2\log n & G=A_n.
\end{cases}
$$

\section{Experimental results}

\subsection{The full attack}
We have implemented our full attack on CBKAP,
and tested it against a large number of parameter settings, including the suggested ones, smaller ones,
much larger ones, and mixed settings (some parameters are small and some are large).
The full attack succeeded to extract the shared key out of the public information correctly, in \emph{all} tested cases,
including those in which the generated subgroup of $S_{n/2}$ was neither $S_{n/2}$ nor $A_{n/2}$.

\subsection{The generic membership search algorithm}
We then moved to a systematic examination of the generic membership search algorithm.
This algorithm worked efficiently and successfully in all experiments, and its time, space
and length of output were all surprisingly close to the estimations computed in the previous
sections.
The most difficult case for this algorithm is where there are only $k=2$ random generators $s_1,s_2$.
Thus, we have made a large battery of experiments for $k=2$.

\forget
\subsubsection{The average number of permutations in Step 1 versus $n^2$}
For various values of $n$, and for $G=S_n$ or $A_n$, we have calculated the average number of permutations considered in Step 1.
Let $f(n)$ be such that $f(n)\cdot n$ is equal to this number. In light of the Minimal Cycle Conjecture \ref{MCC}, one may expect that $f(n)\le n$.
Indeed, we obtained the following values for $f(n)$.

\begin{table}[!htp]
\caption{$f(n)\le n$}\label{alpha}
\begin{tabular}{|c||r|r|r|r|r|r|}
\hline
$n$ & 8 & 16 & 32 & 64 & 128 & 256\\
\hline\hline
$S_n$ & 9.28 & 11.88 & 14.8 & 17.08 & 21.44 & 27.12\\
\hline
$A_n$ & 7.53 & 12.72 & 19.08 & 24.12 & 28.56 & 34.68\\
\hline
\end{tabular}
\end{table}

This may indicate that the actual value in the Minimal Cycle Conjecture is much smaller than $n^2$.
\forgotten

We make the following conventions.
The constant $c$ is $2$ if $G=S_n$, and $3$ if $G=A_n$.
For each $n=8,16,32,64,128,256$, we have conducted at least $1000$ independent experiments altogether.
As $k=2$, in about $750$ of these experiments $\langle s_1,s_2\rangle=S_n$, and
in about $250$, $\langle s_1,s_2\rangle=A_n$. The few cases where neither $S_n$ nor $A_n$ were
generated were ignored.

Each of these many samples suggests a value for the considered parameter. We thus present
the minimum, average, and maximum observed values (with the average boldfaced).


\subsubsection{Step 1}
The upper bound $n^2$ in the Minimal Cycle Conjecture \ref{MCC} turns out to be an over-estimation
for the number of permutations considered in Step 1. Indeed, except for few cases in $n=8$,
\emph{none of our experiments exceeded this bound}.
Thus, we present in Table \ref{step1} the ratio between the number of permutations actually considered in Step 1 and
the estimation $cn$, which is what one would obtained if the permutations were independent.

\mytab{Ratios for the number of permutations in Step 1.}{step1} {
$S_n$
 & $0.06$ & $0.03$ & $0.02$ & $0.01$ & $0$ & $0$\\
 & $\mathbf{2.26}$ & $\mathbf{2.53}$ & $\mathbf{3.47}$ & $\mathbf{5.05}$ & $\mathbf{5.4}$ & $\mathbf{8.55}$\\
 & $112.13$ & $45.88$ & $25.22$ & $102.81$ & $52.62$ & $77.31$\\
\hline $A_n$
 & $0.04$ & $0.02$ & $0.01$ & $0.01$ & $0.01$ & $0$\\
 & $\mathbf{0.51}$ & $\mathbf{0.51}$ & $\mathbf{1.35}$ & $\mathbf{1.28}$ & $\mathbf{2.56}$ & $\mathbf{1.9}$\\
 & $7.63$ & $4.15$ & $15.5$ & $7.73$ & $12.65$ & $17.5$\\
}

\subsubsection{Length of the final expression}

For $k=2$, the average length of the final expression of the given permutation is estimated in
Section \ref{explen} to be, roughly, below
$$\frac{2}{(c-1)\log(2k-1)}\cdot n^2\log n=
\begin{cases}
\frac{2}{\log 3}\cdot n^2\log n & G=S_n\\
\frac{1}{\log 3}\cdot n^2\log n & G=A_n.
\end{cases}
$$
($\log(3)\approx 1.1$).
Table \ref{finlen} shows that this estimation is surprisingly good,
and that in fact, the true resulting length is on average better than
this bound.

\newcommand{\lmathbf}[1]{\mathbf{#1}\else{#1}}

\mytab{Ratios for the length of the final expression.}{finlen}
{

$S_n$
 & $0.07$          & $0.11$          & $0.10$          & $0.08$          & $0.1$           & $0.1$\\
 & $\mathbf{0.31}$ & $\mathbf{0.45}$ & $\mathbf{0.52}$ & $\mathbf{0.62}$ & $\mathbf{0.63}$ & $\mathbf{0.68}$\\
 & $1.14$          & $0.97$          & $0.98$          & $1.06$          & $0.99$          & $0.95$\\
\hline

$A_n$
 & $0.11$         & $0.08$          & $0.08$         & $0.03$          & $0.02$          & $0.03$\\
 & $\mathbf{0.4}$ & $\mathbf{0.4}$ & $\mathbf{0.53}$ & $\mathbf{0.54}$ & $\mathbf{0.62}$ & $\mathbf{0.59}$\\
 & $0.78$         & $0.87$           & $1.07$        & $0.86$          & $0.9$           & $0.87$\\
}

The actual lengths of the expressions produced for the given permutations are given
in Table \ref{aclen}. For clarity, the average lengths are rounded to the nearest integer.

\newcommand{\spc}{\hspace{-0.8pt}}
\newcommand {\sixs}[6]{{#1\spc#2\spc#3\spc#4\spc#5\spc#6\spc}}
\newcommand {\fives}[5]{{#1\spc#2\spc#3\spc#4\spc#5\spc}}
\newcommand {\fours}[4]{{#1\spc#2\spc#3\spc#4\spc}}
\newcommand {\threes}[3]{{#1\spc#2\spc#3\spc}}
\newcommand {\twos}[2]{{#1\spc#2\spc}}

\mytab{Expression lengths using the generic membership search algorithm.}{aclen}
{
$S_n$
 & $16$ & $148$ & $674$ & $2603$ & $14357$ & $65063$\\
 & $\mathbf{\twos76}$ & $\mathbf{\threes580}$ & $\mathbf{\fours3331}$ & $\mathbf{\fives19078}$ & $\mathbf{\fives91120}$ & $\mathbf{\sixs450450}$\\
 & $275$ & $1258$ & $6344$ & $33015$ & $143344$ & $631306$\\
\hline
$A_n$
 & $13$ & $54$ & $248$ & $504$ & $1640$ & $9258$\\
 & $\mathbf{\twos48}$ & $\mathbf{\threes261}$ & $\mathbf{\fours1698}$ & $\mathbf{\fives8328}$ & $\mathbf{\fives44739}$ & $\mathbf{\sixs195534}$\\
 & $94$ & $564$ & $3454$ & $13328$ & $65354$ & $286628$\\
}

For comparison with earlier methods, we looked for expressions of permutations as short products, using GAP's \cite{GAP07}
Schreier-Sims based algorithm (division off stabilizer chains), which uses optimizations similar to Minkwitz's \cite{Min98}.
Here, we have $100$ experiments for $S_n$ and $100$ experiments for $A_n$.
Already for $n=32$, the routines went out of memory in about $1/3$ of the cases
for $A_n$, and in about $2/3$ of the cases for $S_n$.
Thus, we also checked $n=24$ and $n=28$ ($n=28$ seems to be the
largest index which the routines handle well).
The resulting lengths are shown in Table \ref{GAPlen}, where $\infty$ means
``out of memory in too many cases''.

\newcommand{\mystab}[3]{\begin{table}[!htp]\caption{#1}\label{#2}\begin{tabular}{|c||r|r|r|r|r|}
\hline $n$ & 8 & 16 & 24 & 28 & 32\\ \hline\hline #3 \hline \end{tabular} \end{table}}

\mystab{Expression lengths using previous heuristics (Schreier-Sims-Minkwitz).}{GAPlen}
{
$S_n$
 & $5$  & $102$ & 432 & $1047$ & $\infty$\\
 & $\mathbf{\twos22}$ & $\mathbf{\threes255}$ & $\mathbf{\fours8039}$ & $\mathbf{\sixs345272}$ & $\infty$\\
 & $4$2 & $418$ & 350846 & $32729135$ & $\infty$\\
\hline
$A_n$
 & $0$  &  $95$ & 549 & $913$ & $\infty$\\
 & $\mathbf{\twos18}$ & $\mathbf{\threes238}$ & $\mathbf{\fours4101}$ & $\mathbf{\fives59721}$ & $\infty$\\
 & $29$ & $413$ & 35447 & $4012292$ & $\infty$\\
}

We can see that Schreier-Sims methods are better than ours only for small values of $n$, and that they are not
applicable for large $n$, where our algorithm is easily applicable.
Also, note the large difference between the minimal and the maximal obtained lengths. Contrast this with the results in Table \ref{aclen}.

\section{Possible fixes of the Algebraic Eraser and challenges}

As we have demonstrated, no choice of the security parameters makes the Algebraic Eraser
immune to the attack presented here, as long as the keys are generated by standard distributions.

A possible fix may be to change the group $S$ into one whose
elements do not have short expressions in terms of its generators.
This may force the attacker to attack the original matrices (whose entries are
Laurent polynomials in the variables $t_i$) directly, using linear algebraic
methods similar to the ones presented here. It is not clear to what
extent this can be done.

The most promising way to foil our attacks, at least on a small fraction of keys,
may be to use very carefully designed distributions, which are far from standard ones.
Following our attack, Dorian Goldfeld and Paul Gunnels devised
a distribution for which
the equations in phase 1 of the attack have a huge number of solutions,
most of which not leading to the correct shared key \cite{GoldfeldGunnels}.

Another option would be to work in semigroups, and use noninvertible matrices.
This may foil the first phase of our attack.

\medskip

The generic membership search algorithm is of interest beyond its applicability to
the Algebraic Eraser. We have demonstrated,
based on our Minimal Cycle Conjecture \ref{MCC},
that this algorithm easily solves instances
with random permutations, in groups of index which is intractable
when using previously known techniques like those in \cite{Min98}.
Our extensive experiments, reported above, support this assertion.

The most interesting direction of extending the present work
is proving the Minimal Cycle Conjecture, even with $O(n^2)$ instead
of our $n^2$. In fact, proving any polynomial bound
would imply that the diameter of $S_n$ is almost always $O(n^2\log n)$, which would improve
considerably the presently known bound $n^7(\log n)^{O(1)}$ on the diameter.
Alexander Hulpke has informed us that our methods are similar to ones used for
constructive recognition of $S_n$ or $A_n$. This connection may be useful
for the proposed analysis.

Finally, we point out that even without changes, our algorithm
applies in many cases not treated here, as the experiments of the full attack reported
above show.

{\small
\subsection*{Acknowledgements}
This research was partially supported by the the Oswald Veblen Fund, 
the Emmy Noether Research Institute for Mathematics, and the Minerva Foundation of Germany.
We thank John Dixon for his proof of Theorem \ref{BD}.
We thank L\'aszl\'o Babai, Dorian Goldfeld, Stephen Miller, \'Ak\-os Seress, and Adi Shamir, for fruitful discussions.
A special thanks is owed to Martin Kassabov, for comprehensive suggestions which improved the presentation
of this paper, and for pointing out to us the present version of Step 2 of our algorithm.
This version seems to be folklore, but we have initially used a less efficient variant.
We also thank Alexander Hulpke and Stefan Kohl for useful information about GAP.
Our full attack on the Algebraic Eraser was implemented using MAGMA \cite{Magma}.
}


\end{document}